\documentclass[reqno,a4paper]{amsart}

\usepackage[foot]{amsaddr}
\usepackage{graphicx,enumerate,nicefrac,color,bm}
\usepackage{stmaryrd}
\usepackage{amsmath, amssymb}
\usepackage{algorithm,algorithmicx,algpseudocode}
\usepackage{cite}
\usepackage{a4wide}
\usepackage{ulem}
\usepackage{color}
\usepackage{xcolor}

\newcommand{\myState}[1]{\State\parbox[t]{\dimexpr\linewidth-\algorithmicindent}{#1\strut}}
\newcommand{\myStateDouble}[1]{\State\parbox[t]{\dimexpr\linewidth-\algorithmicindent-\algorithmicindent}{#1\strut}}

\newcommand{\norm}[1]{\left\|#1\right\|}

\newcommand{\abs}[1]{\left|#1\right|}
\newcommand{\F}{\mathsf{F}}

\renewcommand{\d}{\mathsf{d}}
\newcommand{\T}{\mathsf{T}}
\newcommand{\ds}{\,\d s}
\newcommand{\G}{\mathsf{G}}

\newcommand{\A}{\mathsf{A}}

\newcommand{\dt}{\,\d t}
\newcommand{\uh}{u}
\newcommand{\dx}{\,\d\bm x}

\newcommand{\I}{\mathcal{I}^h}
\newcommand{\jmp}[1]{\left\llbracket#1\right\rrbracket}
\newcommand{\NN}[1]{\left|\!\left|\!\left|#1\right|\!\right|\!\right|}

\newcommand{\ptc}{{\sf PTC}}
\renewcommand{\P}{\tt Poinc}

\newtheorem{theorem}{Theorem}[section]
\newtheorem{lemma}[theorem]{Lemma}
\newtheorem{proposition}[theorem]{Proposition}

\theoremstyle{definition}
\newtheorem{example}[theorem]{Example}

\newtheorem{remark}[theorem]{Remark}

\title[Adaptive Pseudo-Transient-Continuation-Galerkin Methods]{Adaptive Pseudo-Transient-Continuation-Galerkin Methods for Semilinear Elliptic Partial Differential Equations}

\author[M.~Amrein]{Mario Amrein}
\author[T.~P.~Wihler]{Thomas P.~Wihler}
\address{Mathematics Institute, University of Bern, CH-3012 Switzerland}
\email{mario.amrein@hslu.ch}
\email{wihler@math.unibe.ch}

\begin{document}
\normalem
\begin{abstract}
In this paper we investigate the application of pseudo-transient-continuation (\ptc) schemes for the numerical solution of semilinear elliptic partial differential equations, with possible singular perturbations. We will outline a residual reduction analysis within the framework of general Hilbert spaces, and, subsequently, employ the \ptc-methodology in the context of finite element discretizations of semilinear boundary value problems. Our approach combines both a prediction-type \ptc-method (for infinite dimensional problems) and an adaptive finite element discretization (based on a robust \emph{a posteriori} residual analysis), thereby leading to a \emph{fully adaptive \ptc-Galerkin scheme}. Numerical experiments underline the robustness and reliability of the proposed approach for different examples.
\end{abstract}



\maketitle

\section{Introduction and Problem Formulation}

The focus of this paper is on the numerical approximation of semilinear elliptic partial differential equations (PDE), with possible singular perturbations. More precisely, for a fixed parameter~$\varepsilon>0$ (possibly with~$\varepsilon\ll 1$), and a continuously differentiable function $f:\,\mathbb{R}\to\mathbb{R}$, we consider the problem of finding a solution function~$u:\,\Omega\to\mathbb{R}$ which satisfies
\begin{equation}\label{poisson}
\begin{aligned}
-\varepsilon \Delta u &=f(u) \text{ in }  \Omega,\qquad
u=0 \text{ on } \partial \Omega. 
\end{aligned}
\end{equation}
Here, $\Omega\subset\mathbb{R}^d$, with $d=1$ or $d=2$, is an open and bounded 1d interval or a 2d Lipschitz polygon, respectively. Problems of this type appear in a wide range of applications including, e.g., nonlinear reaction-diffusion in ecology and chemical models~\cite{CaCo03,Ed05,Fr08,Ni11,OkLe01}, economy~\cite{BaBu95}, or classical and quantum physics~\cite{BeLi83,St77}. 
From an analysis point of view, semilinear elliptic boundary value problems~\eqref{poisson} have been studied in detail by a number of authors over the last few decades; we refer, e.g., to the monographs~\cite{AmMa06,Ra86,Sm94} and the references therein.
In particular, solutions of~\eqref{poisson} are known to be typically not unique (even infinitely many solutions may exist), and, in the singularly perturbed case, to exhibit boundary layers, interior shocks, and (multiple) spikes. 
The existence of multiple solutions due to the nonlinearity of the problem and/or the appearance of singular effects are challenging issues when solving problems of this type numerically; see, e.g., \cite{RoStTo08,Verhulst}.

\subsubsection*{Linearized Galerkin Methods}
There are, in general, two approaches when solving nonlinear differential equations numerically: Either the nonlinear PDE problem to be solved is first discretized; this leads to a nonlinear algebraic system. Or, alternatively, a local linearization procedure, resulting in a sequence of linear PDE problems, is applied; these linear problems are subsequently discretized by a suitable numerical approximation scheme. We emphasize that the latter approach enables the use of the large body of existing numerical analysis and computational techniques for \emph{linear} problems (such as, e.g., the development of classical residual-based error bounds). The concept of approximating infinite dimensional nonlinear problems by appropriate \emph{linear discretization schemes} has been studied by several authors in the recent past. For example, the approach presented in~\cite{CongreveWihler:15} (see also the work~\cite{GarauMorinZuppa:11,ChaillouSuri:07}) combines fixed point linearization methods and Galerkin approximations in the context of strictly monotone problems. Similarly, in~\cite{AmreinMelenkWihler:16,AmreinWihler:14,AmreinWihler:15} (see also~\cite{El-AlaouiErnVohralik:11}), the nonlinear PDE problems at hand are linearized by an (adaptive) Newton technique, and, subsequently, discretized by a linear finite element method. On a related note, the discretization of a sequence of linearized problems resulting from the local approximation of semilinear evolutionary problems has been investigated in~\cite{AmreinWihlerTime:15}. In all of the works~\cite{AmreinMelenkWihler:16,AmreinWihler:14,AmreinWihler:15,AmreinWihlerTime:15,CongreveWihler:15}, the key idea in obtaining fully adaptive discretization schemes is to provide a suitable interplay between the underlying linearization procedure and (adaptive) Galerkin methods; this is based on investing computational time into whichever of these two aspects is currently dominant.

\subsubsection*{\ptc-Approach}
In contrast to the classical Newton linearization method, the approach to be discussed in this work relies on a pseudo transient continuation procedure (see, e.g.,~\cite[\S6.4]{5} for finite dimensional problems). The basis of this idea is to first interpret any solution $u$ of the nonlinear equation $ \F(u)=~0$, where~$\F$ is a given operator, as a steady state of the initial value problem 
\[
\dot{u}=\F(u),\qquad u(0)=u_{0},
\]
and, then, to discretize the dynamical system in time by means of the backward Euler method. Furthermore, the resulting sequence of nonlinear problems, $u_{n+1}=u_{n}+t_{n}\F(u_{n+1})$, $n\ge 0$, where~$t_n>0$ is a given time step, is linearized with the aid of the Newton method. This scheme is termed \ptc-method. On a local level, i.e., whenever the iteration is close enough to a solution point, the \ptc-method turns into the standard Newton method. Otherwise, if the iteration is far away from a solution point, then the scheme can be interpreted as a continuation method. In a certain sense, the \ptc-method can also be understood as an inexact Newton method. Following the methodology developed in the articles~\cite{AmreinWihler:14,AmreinWihler:15,AmreinWihlerTime:15,CongreveWihler:15}, the present paper employs the idea of combining the \ptc-linearization approach with adaptive~$\mathbb{P}_1$-finite element methods (FEM). Our analysis will proceed along the lines of~\cite[\S6.4]{5}, with the aim to provide an optimal residual reduction procedure in the local linearization process. Moreover, in order to address the issue of devising $\varepsilon$-robust \emph{a posteriori} error estimates for the Galerkin discretizations, we employ the approach presented in~\cite{Verfuerth}.

\subsubsection*{Outline}
The outline of this paper is as follows. In Section \ref{sec:prediction} we study the \ptc-method within the context of general Hilbert spaces, and derive a residual reduction analysis. Subsequently, the purpose of Section~\ref{sc:Well-Posedness-FEM} is the discretization of the resulting sequence of {\em linear} problems by the finite element method, and the development of an $\varepsilon$-robust \emph{a posteriori} error analysis. The final estimate (Theorem~\ref{thm:1}) bounds the residual in terms of the (elementwise) finite element approximation (FEM-error) and the error caused by the linearization of the original problem. Then, in order to define a fully adaptive \ptc-Galerkin scheme, we propose an interplay between the adaptive \ptc-method and the adaptive finite element approach: More precisely, as the adaptive procedure is running, we either perform a \ptc-step in accordance with the suggested prediction strategy (Section~\ref{sec:prediction}) or refine the current finite element mesh based on the {\em a posteriori residual} estimate (Section~\ref{sc:Well-Posedness-FEM}); this is carried out depending on which of the errors (FEM-error or \ptc-error) is  more dominant in the present iteration step. In Section~\ref{sec:numerics} we provide a series of numerical experiments which show that the proposed scheme is reliable and $\varepsilon$-robust for reasonable choices of initial guesses. Finally, we add a few concluding remarks in Section~\ref{sc:concl}.

\subsubsection*{Problem Formulation}
In this paper, we suppose that a (not necessarily unique) solution~$u\in X:=H^1_0(\Omega)$ of~\eqref{poisson} exists; here, we denote by $H^1_0(\Omega)$ the standard Sobolev space of functions in~$H^1(\Omega)=W^{1,2}(\Omega)$ with zero trace on~$\partial\Omega$. Furthermore, signifying by~$X'=H^{-1}(\Omega)$ the dual space of~$X$, and upon defining the map $\F_{\varepsilon}: X\rightarrow X'$ through
\begin{equation}\label{eq:Fweak}
\left \langle \F_{\varepsilon}(u),v \right \rangle :=  \int_{\Omega}{\left\{f(u)v-\varepsilon \nabla u\cdot \nabla v\right\}}\dx \qquad \forall v\in X,
\end{equation}
where $\left\langle\cdot,\cdot\right\rangle$ signifies the dual pairing in~$X'\times X$, the above problem~\eqref{poisson} can be written as a nonlinear operator equation in~$X'$:
\[
\F_\varepsilon(u)=0,
\]
for an unknown zero~$u\in X$. For the purpose of defining the Newton linearization later on in this manuscript, we note that the Fr\'echet-derivative of $\F_{\varepsilon}$ at $u \in X$ is given by  
\[
\left \langle \F'_{\varepsilon}(u)w,v\right \rangle =\int_{\Omega}{\{f'(u)wv-\varepsilon \nabla w\nabla v\}\dx}, \quad v,w,\in X,
\]
where we write~$f'(u)=\partial_u f(u)$. In addition, we introduce the inner product
\[
(u,v)_{X}:=\int_{\Omega}{\{uv+\varepsilon \nabla u \cdot \nabla v\}\dx},\qquad u,v\in V,
\]
with induced norm on $X$ given by
\[
\NN{u}_{\varepsilon,D}:=\Bigl(\varepsilon\norm{\nabla u}_{0,D}^2 +\norm{u}_{0,D}^2 \Bigr)^{\nicefrac{1}{2}},\qquad u\in H^1(D),
\]
where $\|\cdot\|_{0,D}$ denotes the $L^2$-norm on~$D$. Frequently, for~$D=\Omega$, the subindex~`$D$' will be omitted.
Note that, in the case of $ f(u)=-u+g$, with $g \in L^{2}(\Omega)$, i.e., when \eqref{poisson} is linear and strongly elliptic, the norm $ \NN{\cdot}_{\varepsilon,\Omega}$ is a natural energy norm on $X$. As usual, for any $\varphi \in X'$, the dual norm is given by
\[
\norm{\varphi}_{X'}=\sup_{x\in X\setminus\{0\}}{\frac{\left\langle \varphi,x\right \rangle}{\NN{x}_{\varepsilon}}}.
\]

In what follows we shall use the abbreviation $ x\preccurlyeq y $ to mean $x\leq cy $, for a constant $c>0$ independent of the mesh size $h$ and of~$ \varepsilon>0$.

\section{Abstract Framework in Hilbert Spaces}
\label{sec:prediction}
In this section we briefly revisit a possible derivation of the \ptc-scheme. Moreover, following along the lines of \cite{5} we will discuss how residual reduction, based on a \ptc-iteration-scheme, can be achieved within the context of general Hilbert spaces. To this end, let $ X $ be a real Hilbert space with inner product $(\cdot,\cdot)_{X}$ and induced norm $(x,x)_{X}^{\nicefrac{1}{2}}=\norm{x}_{X}$. Furthermore, by $\mathcal{L}(X;X')$, we signify the space of all bounded linear operators from $X$ into $X'$, with norm
\[
\|\mathsf{L}\|_{\mathcal{L}(X;X')}=\sup_{\genfrac{}{}{0pt}{}{x\in X}{\|x\|_X=1}}\|\mathsf{L}(x)\|_{X'},
\]
for any~$\mathsf{L}\in\mathcal{L}(X;X')$.

\subsection{\ptc-Scheme}
We take the view of dynamical systems, i.e., given a possibly nonlinear operator
\[
\F:X\rightarrow X',
\]
we interpret any zero $u_{\infty}\in X$ of~$\F$, i.e., $\F(u_\infty)=0$, as a \emph{steady state} of the dynamical system
\begin{align}\label{eq:dynamical-system}
u(0)=u_{0},\qquad(\dot{u}(t),v)_{X}=\left\langle \F(u(t)),v\right\rangle\quad\forall v\in X, t>0,
\end{align}
where we denote by~$\langle\cdot,\cdot\rangle$ the dual pairing in~$X'\times X$ as before, and~$u_0\in X$ is a given initial guess. More precisely, we suppose that there exists a solution~$u:\,[0,\infty)\to X$ of~\eqref{eq:dynamical-system} with~$\lim_{t\to\infty}u(t)=u_\infty$ in a suitable sense. Then, we discretize \eqref{eq:dynamical-system} in time using the backward Euler method, i.e.,	
\begin{equation}\label{eq:Euler}
(u_{n+1},v)_{X}=(u_{n},v)_{X}+k_{n}\left\langle \F(u_{n+1}),v\right\rangle \quad \forall v\in X,\qquad n\geq 0,
\end{equation}
where $k_{n}>0 $ signifies the (possibly adaptively chosen) temporal step size. Introducing, for~$n\ge 0$, an operator~$\G_n:\,X\to X'$ by
\[
\left\langle \G_n(u),v \right\rangle:=(u-u_{n},v)_{X}-k_{n}\left\langle \F(u),v\right\rangle\quad\forall v\in X,
\]
we see that the zeros of $ \G_n $ define the next update $u_{n+1}$ in \eqref{eq:Euler}. Then, applying Newton's method to $\G_n$ yields a linear equation for an unknown increment~$\widetilde\delta_n\in X$ such that
\[
\left\langle \G_n'(u_{n})\widetilde\delta_{n},v \right\rangle=-\left\langle \G_n(u_{n}),v \right\rangle\quad\forall v\in X,
\]
and the update
\[
u_{n+1}=u_{n}+\widetilde\delta_{n},
\]
where~$\G_n'$ denotes the Fr\'echet derivative of~$\G_n$. Equivalently, upon rescaling~$\delta_n=k_n^{-1}\widetilde\delta_n$, we have
\begin{equation}
\label{eq:linear-implicit-scheme}
(\delta_{n},v)_{X}-k_{n}\left\langle \F'(u_{n})\delta_{n},v \right\rangle=\left\langle \F(u_{n}),v \right\rangle,\qquad u_{n+1}=u_{n}+k_{n}\delta_{n}.
\end{equation}
Incidentally, with~$\delta_n\rightharpoonup 0$ weakly in~$X$, we obtain Newton's method as applied to~$\F$. In order to simplify notation we introduce, for given~$u\in X$ and~$t>0$, an additional operator
\[
\A[t;u]:\,X \to X',
\]
which is defined by 
\begin{equation}
\label{eq:operator-A}
x\in X:\qquad \left\langle \A[t;u]x,v \right\rangle:=(x,v)_{X}-t\left\langle \F'(u)x,v \right \rangle\quad\forall v\in X.
\end{equation}
We can then rewrite~\eqref{eq:linear-implicit-scheme} as 
\begin{equation}
\label{eq:ptc}
\left\langle \A[k_n;u_{n}]\delta_n,v \right\rangle=\left\langle \F(u_{n}),v \right\rangle\quad\forall v\in X,\qquad u_{n+1}=u_{n}+k_n\delta_n.
\end{equation}
For~$n=0,1,2,\ldots$, with a given initial guess~$u_0\in X$, this iteration defines the~\ptc-scheme for the approximation of a zero of~$\F$. Evidently, in order to be able to solve for $\delta_n $ in \eqref{eq:ptc}, the operator $\A[k_n;u_{n}] $ needs to be invertible.

\subsection{Residual Analysis}
The aim of this section is to derive a residual estimate which paves the way for a residual reduction time stepping strategy. It is based on the following structural assumptions on the derivative of~$\F$:
\begin{enumerate}[(a)]
\item For given~$u_0\in X$, there exists a constant~$\mu=\mu(u_0)>0$ such that
\begin{equation}
\label{eq:A1}\tag{A.1}
\sup_{\genfrac{}{}{0pt}{}{x\in X}{\|x\|_X=1}}\left \langle \F'(u_0)x,x \right \rangle \leq -\mu.
\end{equation}
\item There is a constant~$L\ge 0$ such that there holds the Lipschitz property
\begin{equation}
\label{eq:A2}\tag{A.2}
\norm{\F'(x)-\F'(y)}_{\mathcal{L}(X;X')}\leq L\norm{x-y}_{X}\qquad\forall x,y\in X. 
\end{equation} 
\end{enumerate}
\begin{proposition}\label{pr:lm}
Let~$u_0\in X$ such that $\F'(u_0)\in\mathcal{L}(X;X')$. If~\eqref{eq:A1} holds, and if~$\F(u_0)\in X'$, then the linear problem
\begin{equation}\label{eq:sys}
\A[t;u_0](u(t)-u_0)=t\F(u_0)
\end{equation}
has a unique solution~$u(t)\in X$ for any~$t>0$.
\end{proposition}

\begin{proof}
We apply the Lax-Milgram Lemma. In particular, we show that~$\A[t;u_0]$ is coercive and bounded on~$X$. Indeed, for all~$v\in X$, we have
\begin{equation}\label{eq:coercive}
\langle\A[t;u_0]v,v\rangle
=(v,v)_{X}-t\left\langle \F'(u_0)v,v \right \rangle\ge (1+\mu t)\|v\|_X^2,
\end{equation}
which proves coercivity. Moreover, for~$v,w\in X$, we have
\begin{align*}
\abs{\langle\A[t;u_0]v,w\rangle}
&\le\|v\|_X\|w\|_X+t\|\F'(u_0)v\|_{X'}\|w\|_X.
\end{align*}
Since~$\F'(u_0)$ is bounded, we deduce the boundedness of~$\A[t;u_0]$. This completes the proof.
\end{proof}

In order to devise a residual reduction analysis, we insert two preparatory results.

\begin{lemma}\label{lm:Lemma1}
If~\eqref{eq:A1} is satisfied, then we have
\begin{equation}\label{eq:2nd-estimate}
\norm{\A[t;u_0]^{-1}}_{\mathcal{L}(X';X)}\leq \frac{1}{1+t\mu}.
\end{equation}
Moreover, if~$\F(u_0)\in X'$, the estimate
\begin{equation}
\label{eq:first-estimate}
\norm{u(t)-u_0}_{X}\leq \frac{t}{1+t\mu}\norm{\F(u_0)}_{X'}
\end{equation}
holds true, where~ $u(t)$, $t\ge 0$, is the solution from~\eqref{eq:sys}.
\end{lemma}

\begin{proof}
From~\eqref{eq:coercive}, we readily arrive at
\[
\norm{\A[t;u_0]v}_{X'}\geq \norm{v}_{X}(1+t\mu)\quad\forall v \in X,
\]
from which we deduce~\eqref{eq:2nd-estimate}. Furthermore, the second bound results by definition of~$\A[t;u_0]$ in~\eqref{eq:operator-A} with~$v=u(t)-u_0$, and from~\eqref{eq:A1}. Indeed,
\begin{equation}
\label{eq:use1}
\begin{split}
\norm{u(t)-u_0}_{X}^{2}
&=\langle\A[t;u_0](u(t)-u_0),u(t)-u_0\rangle+t\langle\F'(u_0)(u(t)-u_0),u(t)-u_0\rangle\\
&=t\left \langle \F(u_0),u(t)-u_0\right \rangle+t\left\langle \F'(u_0)(u(t)-u_0),u(t)-u_0\right\rangle \\
&\leq t\norm{\F(u_0)}_{X'}\norm{u(t)-u_0}_{X}-t\mu\norm{u(t)-u_0}_{X}^{2},
\end{split}
\end{equation}
which immediately implies~\eqref{eq:first-estimate}. 
\end{proof}

\begin{lemma}\label{lm:Lemma2}
If~$u$ from~\eqref{eq:sys} is differentiable for any~$t\ge 0$, then it holds that
\begin{equation}\label{eq:id1}
\A[t;u_0]\dot{u}(t)=\F(u_0)+\F'(u_0)(u(t)-u_0),
\end{equation}
as well as
\begin{equation}
\label{eq:derivative-control}
t\A[t;u_{0}]\dot{u}(t)
=(u(t)-u_0,\cdot)_{X}
\end{equation}
in~$X'$.
\end{lemma}

\begin{proof}
Recalling~\eqref{eq:operator-A}, we observe that
\[
\frac{\d}{\dt}\A[t;u_{0}]v=-\F'(u_0)v\qquad\forall v\in X,
\]
in~$X'$. Then, differentiating~\eqref{eq:sys} with respect to $t$ implies
\[
\A[t;u_0]\dot{u}(t)-\F'(u_0)(u(t)-u_0) = \F(u_0),
\]
which yields~\eqref{eq:id1}. Furthermore, multiplying this equality by~$t$, and applying the definition of~$\A[t;u_0]$ from~\eqref{eq:operator-A}, it follows that
\[
t\A[t;u_0]\dot{u}(t)
=t\F(u_0)+(u(t)-u_0,\cdot)_X-\A[t;u_0](u(t)-u_0)
\]
in~$X'$. Using~\eqref{eq:sys} gives~\eqref{eq:derivative-control}.
\end{proof}

Following along the lines of~\cite{5} there holds the ensuing residual reduction result.

\begin{theorem}\label{thm:res}
Under the assumptions in Lemmas~\ref{lm:Lemma1} and~\ref{lm:Lemma2}, and if~\eqref{eq:A2} holds, then we have
\begin{equation}\label{eq:res}
\norm{\F(u(t))}_{X'}\leq \gamma(t)\norm{\F(u_{0})}_{X'},
\end{equation}
with
\[
\gamma(t):=\frac{1}{1+t\mu}\left(1+\frac{Lt^{2}}{2(1+t\mu)}\norm{\F(u_0)}_{X'}\right)>0,
\]
for any~$t\ge 0$.
\end{theorem}

\begin{proof}
For~$t\ge 0$, there holds
\begin{align*}
\F(u(t))&=\F(u_{0})+\int_{0}^{t}{\F'(u(s))\dot{u}(s)\ds}\\
&=\F(u_{0})+\F'(u_0)(u(t)-u_0)+\int_{0}^{t}{(\F'(u(s))-\F'(u_0))\dot{u}(s)\ds}.
\end{align*}
Involving~\eqref{eq:id1}, we infer that
\begin{align*}
\F(u(t))
&=\A[t;u_0]\dot{u}(t)+\int_{0}^{t}{(\F'(u(s))-\F'(u_0))\dot{u}(s)\ds}.
\end{align*}
Therefore, we have
\[
\norm{\F(u(t))}_{X'} \leq \norm{\A[t;u_0]\dot{u}(t)}_{X'}+\int_{0}^{t}{\norm{(\F'(u(s))-\F'(u_0))\dot{u}(s)}_{X'}\ds}.
\]
Employing \eqref{eq:A2} and applying~\eqref{eq:derivative-control}, we arrive at
\begin{equation}
\label{eq:step1}
t\norm{\F(u(t))}_{X'}\leq \norm{u(t)-u_0}_{X}+Lt\int_{0}^{t}{\norm{u(s)-u_0}_{X}\norm{\dot{u}(s)}_{X}\ds}.
\end{equation}
Moreover, again from~\eqref{eq:derivative-control}, we notice that
\begin{equation}\label{eq:aux3}
s\norm{\dot{u}(s)}_{X}\leq	\norm{\A[s;u_0]^{-1}}_{\mathcal{L}(X';X)}\norm{u(s)-u_0}_{X},\qquad s\ge 0,
\end{equation}
and, hence, by virtue of Lemma~\ref{lm:Lemma1}, we obtain
\[
s\norm{u(s)-u_0}_{X}\norm{\dot{u}(s)}_{X}
\le \frac{1}{1+s\mu}\|u(s)-u_0\|_X^2
\leq \frac{s^2}{(1+s\mu)^3}\norm{\F(u_0)}_{X'}^{2}.
\]
Combining this with~\eqref{eq:step1}, and using Lemma~\ref{lm:Lemma1} once more, leads to 
\begin{align*}
t\norm{\F(u(t))}_{X'}&\leq \frac{t}{1+t\mu}\norm{\F(u_0)}_{X'}+Lt\norm{\F(u_0)}_{X'}^2\int_{0}^{t}{\frac{s}{(1+s\mu)^{3}}\ds}\\
&=\frac{t}{1+t\mu}\norm{\F(u_0)}_{X'}+\frac{Lt^3}{2(1+\mu t)^2}\norm{\F(u_0)}_{X'}^2.
\end{align*}
This completes the proof.
\end{proof}

\begin{remark}
Referring to~\eqref{eq:use1}, we see that 
\[
\norm{u(t)-u_0}_{X}^{2}\leq t\left\langle \F(u_0),u(t)-u_0\right\rangle - t\mu \norm{u(t)-u_0}_{X}^{2}.
\]
Hence, whenever there holds $ t\left\langle \F(u_0),u(t)-u_0\right\rangle \leq  \norm{u(t)-u_0}_{X}^{2}$, it follows that~$ \mu \leq 0 $ (as long as there is~$t>0$ with~$u(t)\neq u_0$). In particular, assumption \eqref{eq:A1} is not fulfilled in this case. We may therefore assume that 
\[
\norm{u(t)-u_0}_{X}^{2} < t\left\langle \F(u_0),u(t)-u_0\right\rangle,
\]
for~$t>0$, and~$u(t)\neq u_0$.
\end{remark}

From~\eqref{eq:res} it follows that the residual decreases, i.e., $\|\F(u(t))\|_{X'}<\|F(u_0)\|_{X'}$, if~$\gamma(t)\in(0,1)$. For~$t>0$, this happens if there holds
\begin{equation}\label{eq:t}
\left(\frac{L}{2}\norm{\F(u_{0})}_{X'}-\mu^2\right)t<\mu\qquad (t>0).
\end{equation}
Therefore, if
\[
\frac{L}{2}\norm{\F(u_{0})}_{X'}\le \mu^2,
\]
then any value of~$t>0$ will lead to a reduction of the residual. Otherwise, \eqref{eq:t} can be satisfied as long as~$t$ is chosen sufficiently small; in the special case that $L\norm{\F(u_0)}_{X'}>\mu^2 $, it is elementary to verify that $\gamma(t) $ attains its minimum for
\[
t^\star=\frac{\mu}{L\norm{\F(u_0)}_{X'}-\mu^2}.
\]

\subsection{Pseudo Time Stepping}
In terms of the \ptc-scheme~\eqref{eq:ptc}, for~$n\ge 0$, our previous discussion translates into
\[
\|\F(u_{n+1})\|_{X'}\le\gamma_n\|\F(u_n)\|_{X'},
\]
with a reduction constant
\[
\gamma_n=\frac{1}{1+k_n\mu}\left(1+\frac{Lk_n^{2}}{2(1+k_n\mu)}\norm{\F(u_n)}_{X'}\right)>0;
\] 
cf. Theorem~\ref{thm:res}. If 
\[
\frac{L}{2}\norm{\F(u_{n})}_{X'}\le\mu^2,
\]
then any choice of~$k_n>0$ will imply that~$\gamma_n\in(0,1)$. Otherwise, for~$k_n$ sufficiently small so that
\[
\left(\frac{L}{2}\norm{\F(u_{0})}_{X'}-\mu^2\right)k_n<\mu,
\]
it holds that~$\gamma_n<1$. In particular, if~$L\norm{\F(u_n)}_{X'}>\mu^2$, then
\begin{equation}\label{eq:k*}
k_n^\star=\frac{\mu}{L\norm{\F(u_n)}_{X'}-\mu^2}
\end{equation}
results in a minimal value of~$\gamma_n$. For this value of~$k_n$, we apply~\eqref{eq:first-estimate} to infer the bound
\[
\norm{u_{n+1}-u_n}_{X}\leq \frac{k_n^\star}{1+k_n^\star\mu}\norm{\F(u_n)}_{X'}=\frac{\mu}{L}.
\]
Letting~$\delta_n=\nicefrac{(u_{n+1}-u_n)}{k_n^\star}$ be the increment in the \ptc-iteration~\eqref{eq:linear-implicit-scheme}, this leads to~$\norm{\delta_n}_X\le \nicefrac{\mu}{(k{_n^\star}L)}$, and, therefore,
\begin{equation}
\label{eq:computable}
k_n^\star \leq \frac{\mu}{L\norm{\delta_n}_{X}}.	
\end{equation}
This upper bound does not contain any dual norms, and can, thus, be employed as an approximation of~$k_n^\star$ in practice.

\begin{remark}
In an effort to replace~\eqref{eq:computable} by a computationally even more feasible bound (not involving the possibly unspecified constants~$\mu$ and~$L$), we proceed again along the lines of~\cite{5}. As in~\eqref{eq:use1}, for~$k_n>0$, we have
\[
\|\delta_n\|^2_X\le \langle\F(u_n),\delta_n\rangle-\mu k_n\|\delta_n\|^2_X.
\]
This motivates to define the computable quantity
\[
\bm{\mu}_n:=\frac{\left\langle \F(u_n),\delta_n\right\rangle-\norm{\delta_n}_{X}^2}{k_n\norm{\delta_n}_{X}^2}\geq \mu>0.
\]
Furthermore, similarly as in the proof of Theorem~\ref{thm:res}, we note that
\[
\langle\F(u_{n+1}),\delta_n\rangle=\|\delta_n\|_X^2+\int_{t_{n}}^{t_{n+1}}\langle(\F'(u(s))-\F'(u_n))\dot{u}(s),\delta_n\rangle\ds,
\]
where, for~$i\ge 1$, we let~$t_i=\sum_{j=0}^{i-1} k_j$. Then, by means of~\eqref{eq:A2}, it follows that
\begin{align*}
\frac{|\left\langle \F(u_{n+1}),\delta_n \right\rangle -\norm{\delta_n}_{X}^2|}{\norm{\delta_n}_{X}}
&\leq \int_{t_{n}}^{t_{n+1}}{\norm{(\F'(u(s))-\F'(u_n))\dot{u}(s)}_{X'}\ds}\\
&\leq L \int_{t_{n}}^{t_{n+1}}{\norm{u(s)-u_n}_{X}\norm{\dot{u}(s)}_{X}\ds}.
\end{align*}
Furthermore, using~\eqref{eq:aux3} and employing~Lemma~\ref{lm:Lemma1}, this transforms into
\begin{align*}
\frac{|\left\langle \F(u_{n+1}),\delta_n \right\rangle -\norm{\delta_n}_{X}^2|}{\norm{\delta_n}_{X}}
&\leq L\int_{t_{n}}^{t_{n+1}}\frac{(s-t_{n})^{-1}}{1+(s-t_n)\mu}\|u(s)-u_n\|_{X}^2\ds.
\end{align*}
Approximating the integral with the aid of the trapezoidal rule, and recalling that~$u_{n+1}-u_n=k_n\delta_n$, cf.~\eqref{eq:linear-implicit-scheme}, yields
\[
\frac{|\left\langle \F(u_{n+1}),\delta_n \right\rangle -\norm{\delta_n}_{X}^2|}{\norm{\delta_n}_{X}}\le \frac{L}{2}k_n^2\|\delta_n\|_X^2+\mathcal{O}(k_n^4).
\]
We then define
\[
\bm{L}_n:=\frac{2|\left\langle \F(u_{n+1}),\delta_n \right\rangle -\norm{\delta_n}_{X}^2|}{k_n^2\norm{\delta_n}_{X}^3}\leq L+\mathcal{O}(k_n^2).
\]
Replacing~$\mu$ and~$L$ in~\eqref{eq:computable} by~$\bm\mu_n$ and~$\bm L_n$, respectively, we are led to introduce the following pseudo time step
\begin{equation}
\label{eq:stepsize-control}
\bm{k}_n^\star=\frac{k_n}{2}\cdot \abs{\frac{	\left \langle \F(u_n),\delta_n \right \rangle-\norm{\delta_n}_{X}^{2}}{\left \langle \F(u_{n+1}),\delta_n \right \rangle -\norm{\delta_n}_{X}^2}},
\end{equation}
which does not require explicit knowledge on~$\mu$ and~$L$.
\end{remark}


\section{Application to Semilinear Problems}\label{sc:Well-Posedness-FEM}

In this section, we will apply the abstract setting from the previous section to the semilinear problem~\eqref{poisson}, with~$\F=\F_\varepsilon$ from~\eqref{eq:Fweak}.

\subsection{\ptc-Linearization}
For~$u_n\in X$ and~$k_n>0$, the \ptc-method~\eqref{eq:linear-implicit-scheme} is to find~$\delta_n \in X$ such that
\begin{equation}\label{eq:weak-formulation}
a_\varepsilon(u_{n},k_n;\delta_n,v) = \ell_\varepsilon(u_{n};v)\qquad\forall v\in X,
\end{equation}
and~$u_{n+1}=u_n+k_n\delta_n$, where, for {\em fixed}~$u\in X$, $t>0$, we consider the bilinear form
\begin{equation*}
\begin{aligned}
a_{\varepsilon}(u,t;\delta,v)&:=(\delta, v)_{X}
-t\int_{\Omega}{\{f'(u)\delta v-\varepsilon \nabla \delta \cdot \nabla v\}\dx},\qquad\delta,v\in X,
\end{aligned}
\end{equation*}
as well as the linear form
\begin{align*}
\ell_{\varepsilon}(u;v)&:=\int_{\Omega}{\{f(u)v-\varepsilon \nabla u \cdot \nabla v\}\dx},\qquad v\in X.
\end{align*}

Throughout, for given~$u_n$, $n\ge 0$, we assume that \eqref{eq:linear-implicit-scheme} has a unique solution~$u_{n+1}$. In fact, this property can be made rigorous if certain assumptions on the nonlinearity~$f$ are satisfied. This will be addressed in the ensuing two propositions.

\begin{proposition}\label{pr:f}
If~$\sigma_f:=\sup_{x\in\mathbb{R}}f'(x)<\varepsilon C_{\P}^{-2}$, where~$C_{\P}=C_{\P}(\Omega)$ is the constant in the Poincar\'e inequality on $\Omega$,
\begin{equation}
\label{eq:Poincare}
\norm{w}_{0}\leq C_{\P}\norm{\nabla w}_{0}, \quad \forall w \in X,
\end{equation}
then~\eqref{eq:A1} is satisfied with
\[
\mu=\frac{\varepsilon C_{\P}^{-2}-\sigma_f}{\varepsilon C_{\P}^{-2}+1}>0.
\]
\end{proposition}

\begin{proof}
Let us set
\[
\zeta:=\frac{1+\sigma_f}{\varepsilon C_{\P}^{-2}+1}.
\]
By our assumptions, there holds that~$\zeta<1$. Then, with~\eqref{eq:Poincare}, for~$u,v\in X$, we have 
\begin{align*}
\langle\F_\varepsilon'(u)v,v\rangle&=(\zeta-1)\varepsilon\|\nabla v\|_0^2-\zeta\varepsilon\|\nabla v\|_0^2+\int_\Omega f'(u)v^2\dx\\
&\le(\zeta-1)\varepsilon\|\nabla v\|_0^2+\int_\Omega \{f'(u)-\zeta C_{\P}^{-2}\varepsilon\}v^2\dx\\
&\le(\zeta-1)\NN{v}^2_\varepsilon
=-\mu \NN{v}^2_\varepsilon.
\end{align*}
Hence, \eqref{eq:A1} is verified.
\end{proof}

\begin{remark}
Within a given \ptc-iteration, for~$n\ge 0$, the proof of the above result reveals that~$\sigma_f$ can be replaced by the possibly sharper value~$\sigma_f:=\sup_{\Omega}f'(u_n)$.
\end{remark}

\begin{proposition}\label{pr:f2}
If $f'$ is globally Lipschitz continuous with Lipschitz constant $L_{f'}$, that is,
\begin{equation}\label{eq:f'Lip}
|f'(u_1)-f'(u_2)|\le L_{f'}|u_1-u_2|\qquad\forall u_1,u_2\in\mathbb{R},
\end{equation}
then~\eqref{eq:A2} is fulfilled with $L:=CL_{f'}\varepsilon^{-1}$,
where~$C>0$ is a constant only depending on~$\Omega$.
\end{proposition}

\begin{proof}
For $u_1,u_2,w,v \in X$ there holds
\[
|\left\langle (\F'(u_1)-\F'(u_2))w,v \right \rangle| 
\le\norm{(f'(u_1)-f'(u_2))wv}_{L^1(\Omega)}.
\]
Employing~\cite[Lemma~A.1]{AmreinWihler:15}, and applying the Lipschitz continuity~\eqref{eq:f'Lip}, we obtain
\begin{align*}
|\left\langle (\F'(u_1)-\F'(u_2))w,v \right \rangle| 
&\le C\|f'(u_1)-f'(u_2)\|_{0}\|\nabla w\|_0\|\nabla v\|_0\\
&\le CL_{f'}\|u_1-u_2\|_{0}\|\nabla w\|_0\|\nabla v\|_0\\
&\le CL_{f'}\varepsilon^{-1}\NN{u_1-u_2}_\varepsilon\NN{w}_\varepsilon\NN{v}_\varepsilon,
\end{align*}
for a constant~$C>0$ only depending on~$\Omega$. This verifies~\eqref{eq:A2}.
\end{proof}

\begin{remark}
If the assumptions in the above Propositions~\ref{pr:f} and~\ref{pr:f2} are satisfied, and if for given $u_n\in X$ there holds $f(u_{n}) \in L^{2}(\Omega)$, then the linear problem \eqref{eq:weak-formulation} has a unique solution $\delta_{n}\in X$; cf.~Proposition~\ref{pr:lm}. 
\end{remark}

\subsection{\ptc-Galerkin Discretization}
In order to provide a numerical approximation of~\eqref{poisson}, we will discretize the \emph{linear} weak formulation~\eqref{eq:weak-formulation} by means of a finite element method, which, in combination with the \ptc-iteration, constitutes a \ptc-Galerkin approximation scheme. Furthermore, we shall derive {\em a posteriori} residual estimates for the finite element 
discretization which allow for an adaptive refinement of the meshes in each \ptc-step. This, together with the adaptive prediction strategy from Section~\ref{sec:prediction}, leads to a fully adaptive \ptc-Galerkin discretization method for~\eqref{poisson}.

\subsubsection{Finite Element Meshes and Spaces}
Let $ \mathcal{T}^h=\{T\}_{T\in\mathcal{T}^h}$ be a regular and shape-regular mesh partition of $\Omega $ into disjoint open simplices, i.e., any~$T\in\mathcal{T}^h$ is an affine image of the (open) reference simplex~$\widehat T=\{\widehat x\in\mathbb{R}_+^d:\,\sum_{i=1}^d\widehat x_i<1\}$. By~$h_T=\mathrm{diam}(T)$ we signify the element diameter of~$T\in\mathcal{T}^h$, and by $h=\max_{T\in\mathcal{T}^h}h_T$ the mesh size. Furthermore, by $\mathcal{E}^h$ we denote the set of all interior mesh nodes for~$d=1$ and interior (open) edges for~$d=2$ in~$\mathcal{T}^h$. In addition, for~$T\in\mathcal{T}^h$, we let~$\mathcal{E}^h(T)=\{E\in\mathcal{E}^h:\,E\subset\partial T\}$. For~$E\in\mathcal{E}^h$, we let~$h_E$ be the mean of the lengths of the adjacent elements in 1d, and the length of~$E$ in~2d. Let us also define the following two quantities:
\begin{equation}\label{boundary}
\begin{aligned}
\alpha_T&:=\min(1,\varepsilon^{-\nicefrac12}h_T),\qquad
\alpha_E:=\min(1,\varepsilon^{-\nicefrac12}h_E),
\end{aligned}
\end{equation}
for~$T\in\mathcal{T}^h$ and~$E\in\mathcal{E}^h$, respectively.

We consider the finite element space of continuous, piecewise linear functions on $\mathcal{T}^h$ with zero trace on~$\partial\Omega$, given by
\begin{equation*}
V_{0}^{h}:=\{\varphi\in H^1_0(\Omega):\,\varphi|_{T} \in \mathbb{P}_{1}(T) \, \forall T \in \mathcal{T}^h\},
\end{equation*}
respectively, where~$\mathbb{P}_1(T)$ is the standard space of all linear polynomial functions on~$T$. 

\subsubsection{Linear Finite Element Discretization}
For given~$k_n>0$ and $u_n^{h}\in V_{0}^{h}$, $n\ge 0$, we consider the finite element approximation of~\eqref{eq:weak-formulation}, which is to find~$\delta_{n}^h\in V_{0}^h$ such that
\begin{equation}\label{eq:fem}
a_\varepsilon(u_n^h,k_{n};\delta_{n}^h,v)=\ell_\varepsilon(u_n^h;v)\qquad\forall v\in V_{0}^{h};
\end{equation}
for $n=0$, the function~$u_0^{h}\in V_{0}^{h}$ is a prescribed initial guess.
Introducing the linearization operator
\[
\T_{f}(u):=f(u_{n}^{h})+f'(u_{n}^{h})(u-u_{n}^{h}),
\]
as well as
\[
\uh_{n+1}^{h}:=u_{n}^{h}+k_n\delta_{n}^{h},
\]
and rearranging terms, \eqref{eq:fem} can be rewritten as
\begin{equation}
\label{eq:start}
\begin{aligned}
\int_{\Omega}{\varepsilon\nabla \uh_{n+1}^{h} \cdot \nabla v\dx} = \int_{\Omega}{(\T_{f}(u_{n+1}^{h})-\delta_{n}^{h}) v}\dx,
\end{aligned}
\end{equation}
for any~$v  \in V_{0}^h$.

\subsection{{\em A Posteriori} Residual Analysis}
The aim of this section is to derive {\em a posteriori} residual bounds for the linearized FEM~\eqref{eq:fem}.

\subsubsection{\emph{A Posteriori} Residual Bound}
In order to measure the discrepancy between the finite element discretization~\eqref{eq:fem} and the original problem~\eqref{poisson}, a natural quantity 
to bound is the residual~$\F_{\varepsilon}(u_{n+1}^{h})$ in~$X'$. Let $\I:\,H_{0}^{1}(\Omega)\rightarrow V_{0}^{h} $ be the quasi-interpolation operator of Cl\'ement (see, e.g., \cite[Corollary~4.2]{AmreinWihler:15}). Then, testing~\eqref{eq:start} with~$\I v\in V^h_0$, for an arbitrary~$v\in X$, implies that
\[
\int_{\Omega}{\varepsilon\nabla \uh_{n+1}^{h} \cdot \nabla \I v\dx} = \int_{\Omega}{(\T_{f}(u_{n+1}^{h}) -\delta_{n}^{h})\I v\dx}.
\]
Then, there holds the identity
\begin{align*}
\left \langle \F_{\varepsilon}(\uh_{n+1}^{h}),v \right \rangle 
&=\int_\Omega {\varepsilon\nabla \uh_{n+1}^{h}\cdot\nabla (\I v-v)}\dx
+\int_\Omega (\T_{f}(u_{n+1}^h)-\delta_{n}^{h}) (v-\I v)\dx\\
&\quad-\int_\Omega\left\{\T_{f}(u_{n+1}^h)-\delta_{n}^{h}-f(\uh_{n+1}^{h})\right\}v\dx,
\end{align*}
for any~$v\in X$. Integrating by parts in the first term on the right-hand side, recalling the fact that~$(v-\I v)=0$ on~$\partial\Omega$, and applying some elementary calculations, yields that
\[
\left \langle \F_{\varepsilon}(\uh_{n+1}^{h}),v \right \rangle
=\sum_{E\in\mathcal{E}^h}a_E+\sum_{T\in\mathcal{T}^h}(b_T-c_T),
\]
where
\begin{equation}
\begin{aligned}
a_{E}&:= \int_{E}\varepsilon\jmp{\nabla \uh_{n+1}^{h}}(\I v-v)\ds,\qquad c_{T}:=\int_{T}{\left\{\T_{f}(u_{n+1}^h)-\delta_{n}^{h}-f(\uh_{n+1}^{h})\right\}v}\dx,\nonumber \\
b_{T}&:=\int_{T}{\left\{\varepsilon \Delta \uh_{n+1}^{h}+\T_{f}(u_{n+1}^h)-\delta_{n}^{h}\right\}(v-\I v)}\dx,
\end{aligned}
\end{equation}
with~$E\in\mathcal{E}^h$, $T\in\mathcal{T}^h$. Here, for any edge $ E=\partial T^\sharp\cap \partial T^\flat \in \mathcal{E}^h $ shared by two neighboring elements~$T^\sharp, T^\flat\in\mathcal{T}^h$, where $\bm n^\sharp$ and~$\bm n^\flat$ signify the unit outward vectors on~$\partial T^\sharp$ and~$\partial T^\flat$, respectively, we denote by 
\[
\jmp{\nabla u_{n+1}^h}(\bm x)=\lim_{t\to 0^+}\nabla u_{n+1}^h(\bm x+t\bm n^\sharp)\cdot\bm n^\sharp+\lim_{t\to 0^+}\nabla u_{n+1}^h(\bm x+t\bm n^\flat)\cdot\bm n^\flat,\qquad \bm x\in E,
\]
the jump across~$E$. Then, for~$T\in\mathcal{T}^h$, defining the linearization residual
\begin{equation}
\label{linearizationerror}
R_{n,T}:=\norm{\T_{f}(u_{n+1}^h)-\delta_{n}^{h}-f(\uh_{n+1}^{h})}_{0,T},
\end{equation}
as well as the FEM approximation residual
\begin{equation}
\label{Femerror}
\eta_{n,T}^2:= \alpha_{T}^2 \norm{\varepsilon \Delta \uh_{n+1}^{h}+\T_f(u_{n+1}^h)-\delta_{n}^{h}}_{0,T}^2+\frac{1}{2}\sum_{E\in \mathcal{E}^{h}(T)}{\varepsilon^{-\nicefrac12}\alpha_E\norm{\varepsilon \jmp{\nabla \uh_{n+1}^{h} }}_{0,E}^2},
\end{equation}
with~$\alpha_T$ and~$\alpha_E$ from~\eqref{boundary}, we proceeding along the lines of the proof of~\cite[Theorem~4.4]{AmreinWihler:15} in order to obtain the following result.

\begin{theorem}
\label{thm:1}
For~$n\ge 0$ there holds the upper \emph{a posteriori} residual bound
\begin{equation}
\label{eq:upperbound}
\norm{\F(\uh_{n+1}^{h})}_{X'}^2
\preccurlyeq R_{n,\Omega}^2+\sum_{T\in \mathcal{T}^h}{\eta_{n,T}^{2}},
\end{equation}
with~$R_{n,\Omega}$ and~$\eta_{n,T}$, $T\in\mathcal{T}^h$, from~\eqref{linearizationerror} and~\eqref{Femerror}, respectively.
\end{theorem}

\begin{remark}
Following our approach in~\cite[Theorem~4.5]{AmreinWihler:15}, under certain conditions on the nonlinearity~$f$, it can be shown that the right-hand side of the above bound~\eqref{eq:upperbound} is equivalent to the error norm~$\NN{u-u_{n+1}^{h}}_{\varepsilon,\Omega}$. 
\end{remark}

\begin{remark}
In addition to the upper estimate in the above Theorem~\ref{thm:1}, we notice that local lower \emph{a posteriori} residual bounds can be established for the proposed \ptc-Galerkin method as well. Indeed, this can be accomplished similarly to our analysis in~\cite[\S4.4.2]{AmreinWihler:15} (see also~\cite{Verfuerth}), which is based on the application of standard bubble function techniques.
\end{remark}
	
\subsection{A Fully Adaptive \ptc-Galerkin Algorithm}
We will now propose a procedure that will combine the \ptc-method presented in~Section~\ref{sec:prediction} with an automatic finite element mesh refinement strategy. More precisely, based on the \emph{a posteriori} residual bound from Theorem~\ref{thm:1}, the main idea of our approach is to provide an interplay between \ptc-iterations and adaptive mesh refinements which is based on monitoring the two residuals in~\eqref{linearizationerror} and~\eqref{Femerror}, and on acting according to whatever quantity is dominant in the current computations. We make the assumption that the \ptc-Galerkin sequence $\left\{u_{n+1}^{h}\right\}_{n\ge 0} $ given by~\eqref{eq:fem}, with step size~$k_n^\star$ from~\eqref{eq:k*} (or~$\bm{k}^\star_n$ from \eqref{eq:stepsize-control}) is well-defined as long as the iterations are being performed. The individual computational steps are summarized in Algorithm~\ref{al:full}.

\begin{algorithm}
\caption{Fully-adaptive \ptc-Galerkin method}
\label{al:full}
\begin{algorithmic}[1]
\State Given a parameter~$\theta>0$, a (coarse starting) triangulation $ \mathcal{T}^h $ of~$\Omega$, an initial step size $k_0>0$, a maximal number of degrees of freedom $\text{DOF}_{\text{max}}$, and an initial guess $ u_{0}^{h} \in  V_{0}^{h} $. Set~$n\gets 0$.
\While {$\text{DOF}\le\text{DOF}_{\text{max}}$} 
\myState {Compute the FEM solution~$u_{n+1}^h$ from~\eqref{eq:fem} on the mesh~$\mathcal{T}^h$.}
\myState {Evaluate the corresponding residual indicators $ \eta_{n,T} $, $ T \in \mathcal{T}^h $, and $R_{n,\Omega} $ from~\eqref{linearizationerror} and~\eqref{Femerror}, respectively.
}
\If{%
\[
R_{n,\Omega}^2\le \theta\sum_{T\in \mathcal{T}^h}{\eta_{n,T}^2}
\]
\quad}
\myStateDouble{refine the mesh $ T \in \mathcal{T}^h $ adaptively based on the elementwise residual indicators~$\eta_{n,T}$, $T\in\mathcal{T}^h$ from Theorem~\ref{thm:1}, and go back to step~({\footnotesize 2:}) with the previously computed solution~$u_{n+1}^{h}$ as interpolated on the refined mesh;}
\Else \myStateDouble{perform another \ptc-step based on the new step size $k_n=\bm{k}_n^\star$ as proposed in~\eqref{eq:stepsize-control} and go back to~({\footnotesize 3:}).}
\EndIf
\myState{set~$n\leftarrow n+1$.}
\EndWhile
\end{algorithmic}
\end{algorithm}
\subsection{Numerical Experiments}
\label{sec:numerics}
We will now illustrate and test the above fully adaptive Algorithm~\ref{al:full} with two numerical experiments in 2d. The linear systems resulting from the finite element discretization~\eqref{eq:start} are solved by means of a direct solver.

\begin{example}\label{ex:1}
Let us consider first the Sine-Gordon type problem
\begin{equation*}
-\varepsilon\Delta u = -\sin(u)-u+1,\ \text{in }\Omega=(0,1)^2,\qquad
u=0\ \text{on }\partial\Omega.
\end{equation*}
Here, $f(u)=-\sin(u)-u+1$, and~$f'(u)=-\cos(u)-1$. In particular, by application of Proposition~\ref{pr:f}, we observe that the structural assumptions~\eqref{eq:A1} and~\eqref{eq:A2} are fulfilled.
Neclecting the boundary conditions for a moment, one observes that the unique positive zero $ u \approx 0.51$ of $ f(u) $ is a solution of the PDE. We therefore expect boundary layers along $\partial \Omega$; see Figure \ref{Sine-Gordon-Mesh} (right). Moreover, the focus of this experiment is on the robustness of the \emph{a posteriori} residual bound \eqref{eq:upperbound} with respect to the singular perturbation paramater $\varepsilon$ as 
$\varepsilon \to 0$. Starting from the initial mesh depicted in Figure \ref{Sine-Gordon-Mesh} (left) with $u_{0}^{h}(\nicefrac{1}{2},\nicefrac{1}{2})\approx \nicefrac{1}{2}$, we test the fully adaptive \ptc-Galerkin Algorithm \ref{al:full} for different choices of $\varepsilon=\{10^{-i}\}_{i=0}^{9}$. 
In Algorithm \ref{al:full} the parameters are chosen to be $\theta = 0.5$ and $ k_0=1$. As $\varepsilon\to 0 $ the resulting solutions feature ever stronger boundary layers; see Figure \ref{Sine-Gordon-Mesh} (right). The performance data in Figure 
\ref{Performance-Data} (left) shows that the residuals decay, firstly, robust in $\varepsilon$, and, secondly, of (optimal) order~$\nicefrac{1}{2}$ with respect to the number of degrees of freedom. 
\end{example}

\begin{figure}
\includegraphics[width=0.43\textwidth]{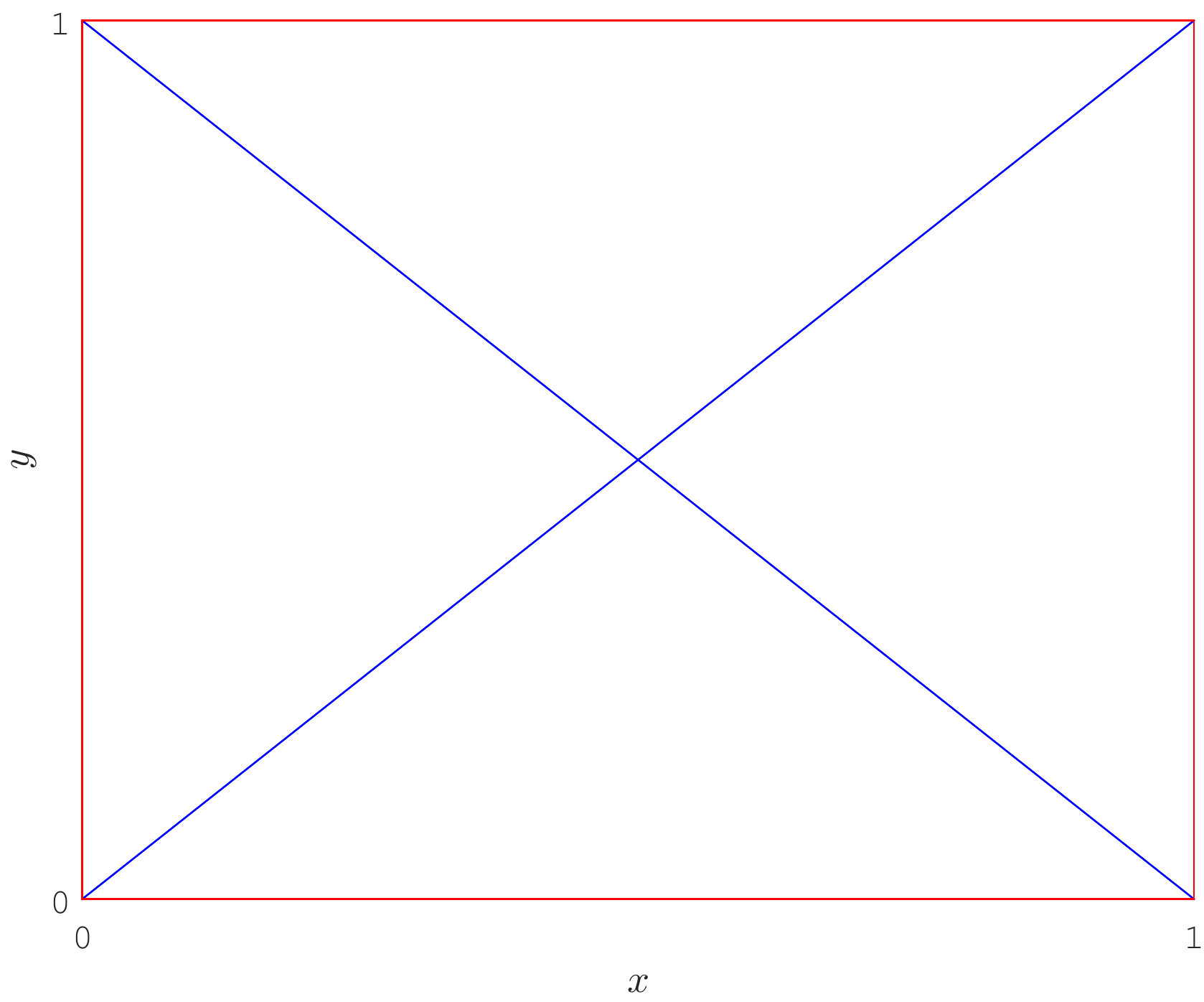}
\hfill
\includegraphics[width=0.43\textwidth]{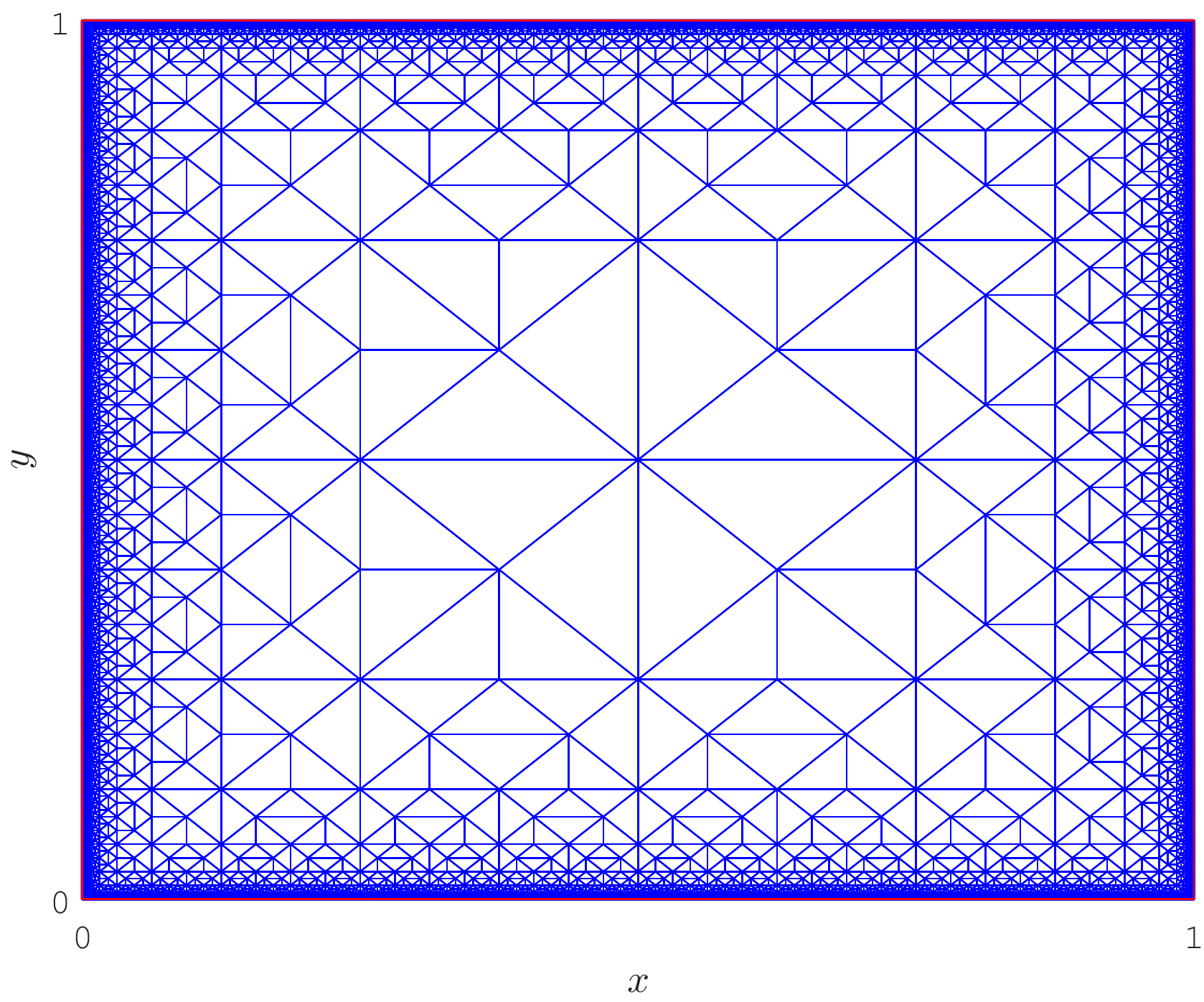}
\caption{Example~\ref{ex:1} for $\varepsilon=10^{-7}$: Initial mesh (left), and the adaptively refined mesh resolving the solution (right).}
\label{Sine-Gordon-Mesh}
\end{figure}

\begin{figure}
\includegraphics[width=0.45\textwidth]{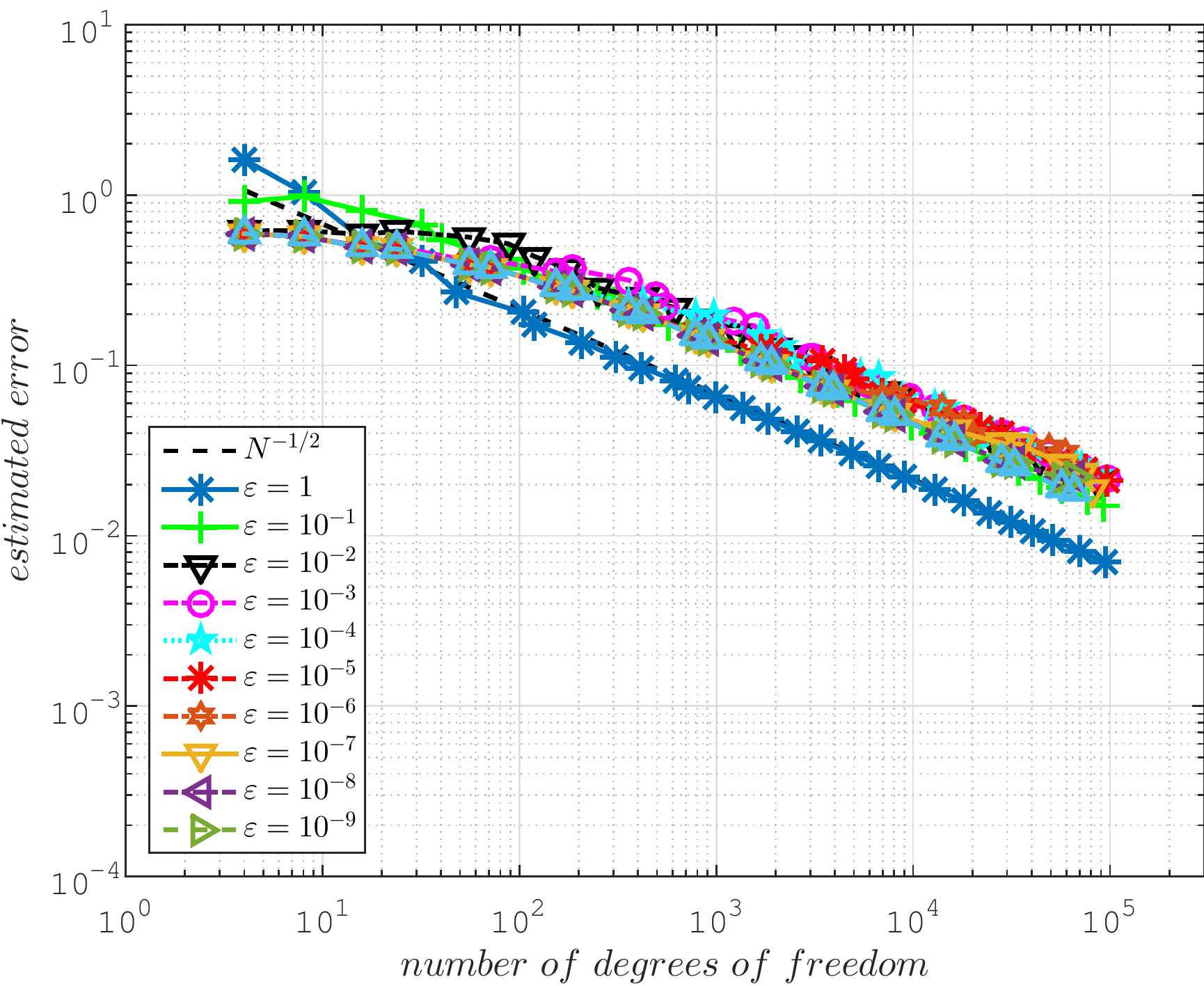}
\hfill
\includegraphics[width=0.45\textwidth]{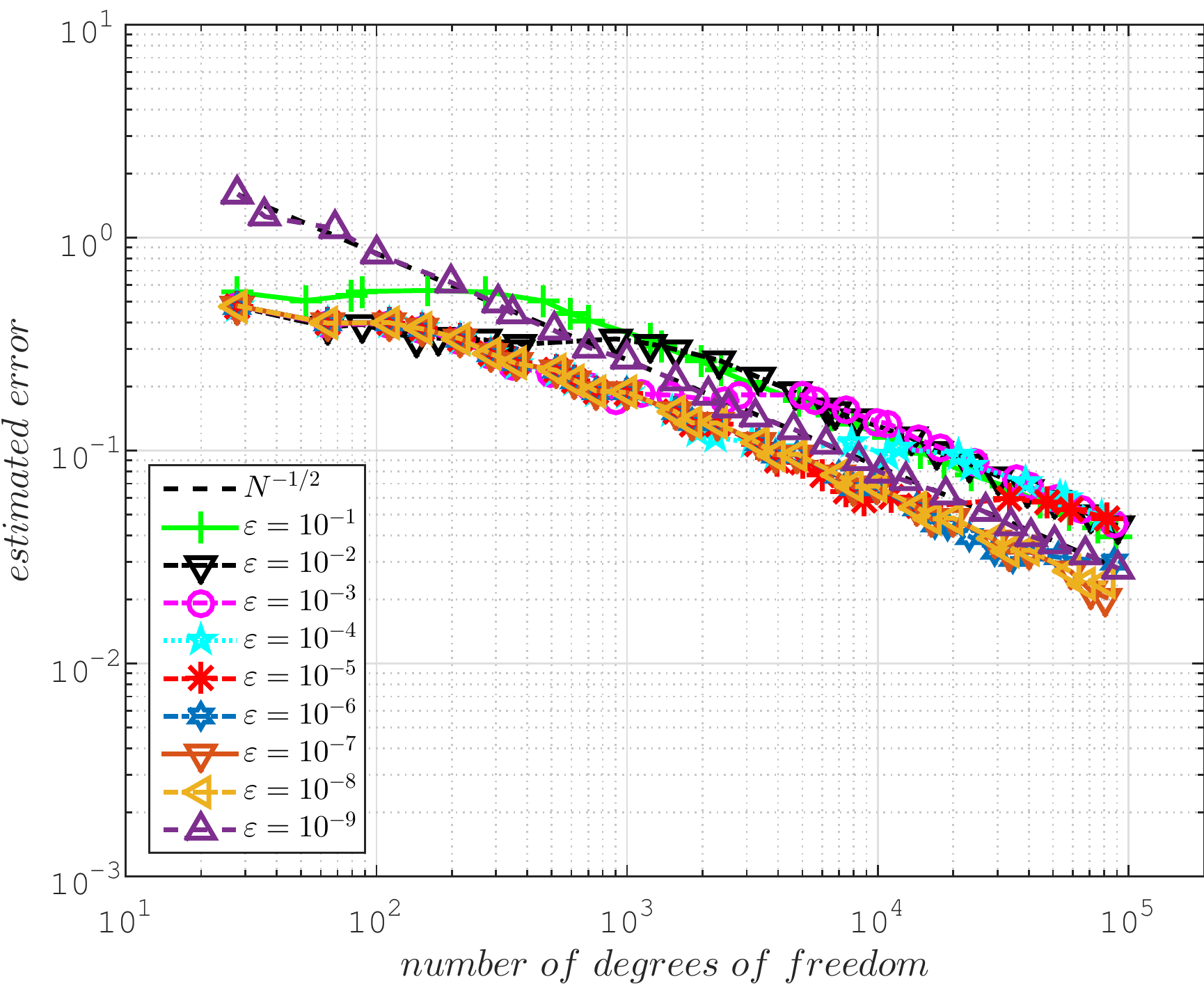}
\caption{Estimated residuals for different choices of $\varepsilon$. On the left for Example~\ref{ex:1} and on the right for Example~\ref{ex:2}.}
\label{Performance-Data}
\end{figure}
\begin{example}

\label{ex:2}
Finally, we turn to the well-known nonlinear Ginzburg-Landau equation on the square~$\Omega=(-1,1)^2$ given by
\begin{equation*}
\begin{aligned}
-\varepsilon\Delta u&= u(1- u^2) \ \text{in } \Omega,\qquad
u = 0 \ \text{on } \partial \Omega.
\end{aligned}
\end{equation*}
Clearly $ u\equiv 0 $ is a solution. In addition, any solution~$u$ appears pairwise as $-u$ is obviously a solution also. Again, neglecting the boundary conditions for a moment, we observe that $u\equiv1$ and $u \equiv-1$ are solutions of the PDE. We therefore expect boundary layers along $ \partial \Omega$, and possibly within the domain~$\Omega$; see Figure~\ref{Ginzburg-Landau-Mesh} (right).
Here we always start from the initial mesh depicted in Figure~\ref{Ginzburg-Landau-Mesh} (left) with $u_{0}^{h}\equiv 1$ on the interior nodes. Again we test the fully adaptive \ptc-Galerkin Algorithm~\ref{al:full} for different choices of $\varepsilon=\{10^{-i}\}_{i=0}^{9}$.
The parameters are still chosen to be $\theta = 0.5$ and $ k_0=1$. As in Example~\ref{ex:1}, for $\varepsilon\to 0 $ the resulting solution feature ever stronger boundary layers; see Figure \ref{Ginzburg-Landau-Mesh} (right). In addition, from the performance data given in Figure 
\ref{Performance-Data} (right) we observe that the residuals decay again robust in $\varepsilon$. Finally we notice convergence of (optimal) order~$\nicefrac12$ with respect to the number of degrees of freedom. We remark that, although~\eqref{eq:A1} and~\eqref{eq:A2} are not necessarily satisfied for this problem, our fully adaptive \ptc-Galerkin approach still delivers good results.
\begin{figure}
\includegraphics[width=0.43\textwidth]{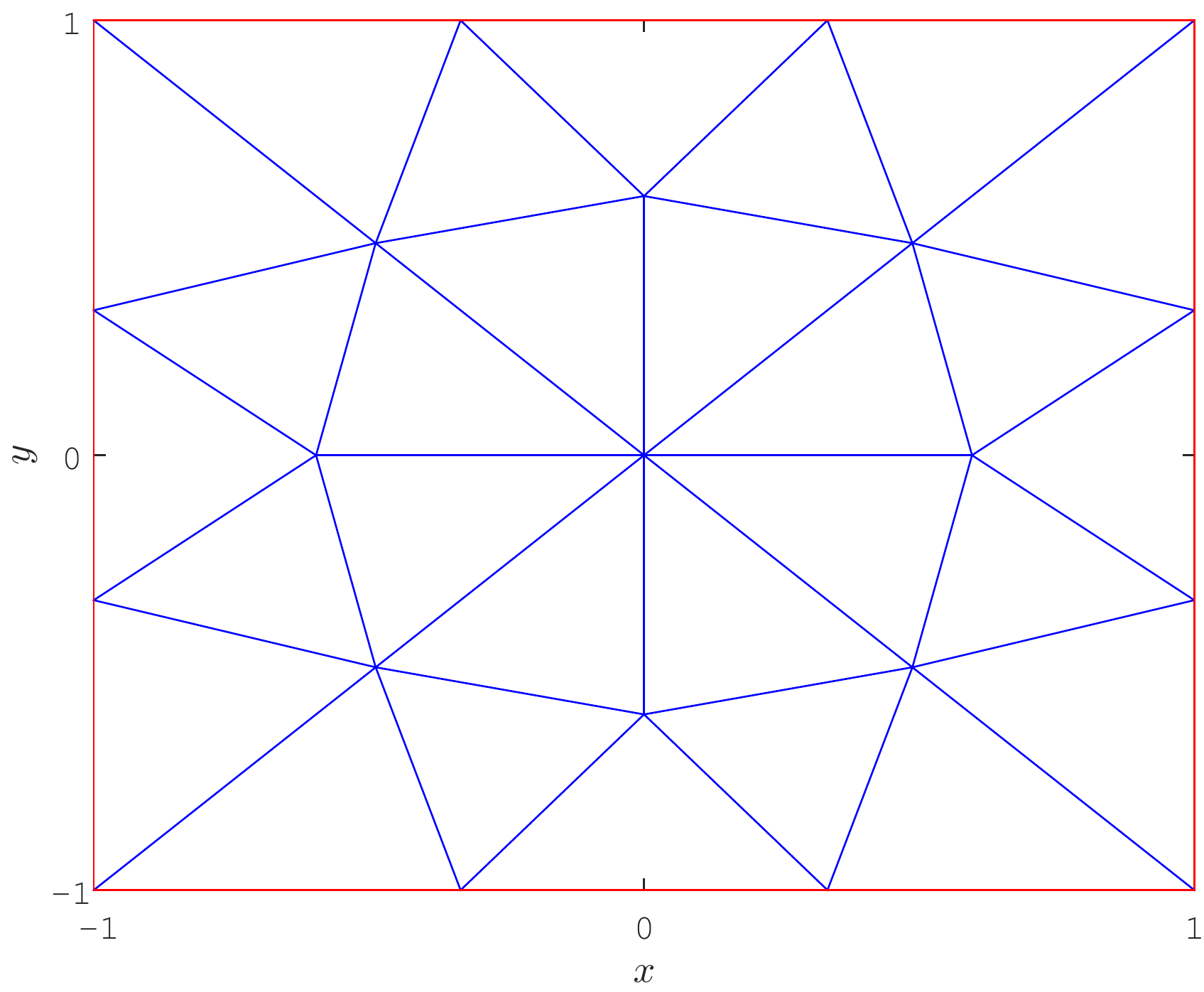}
\hfill
\includegraphics[width=0.43\textwidth]{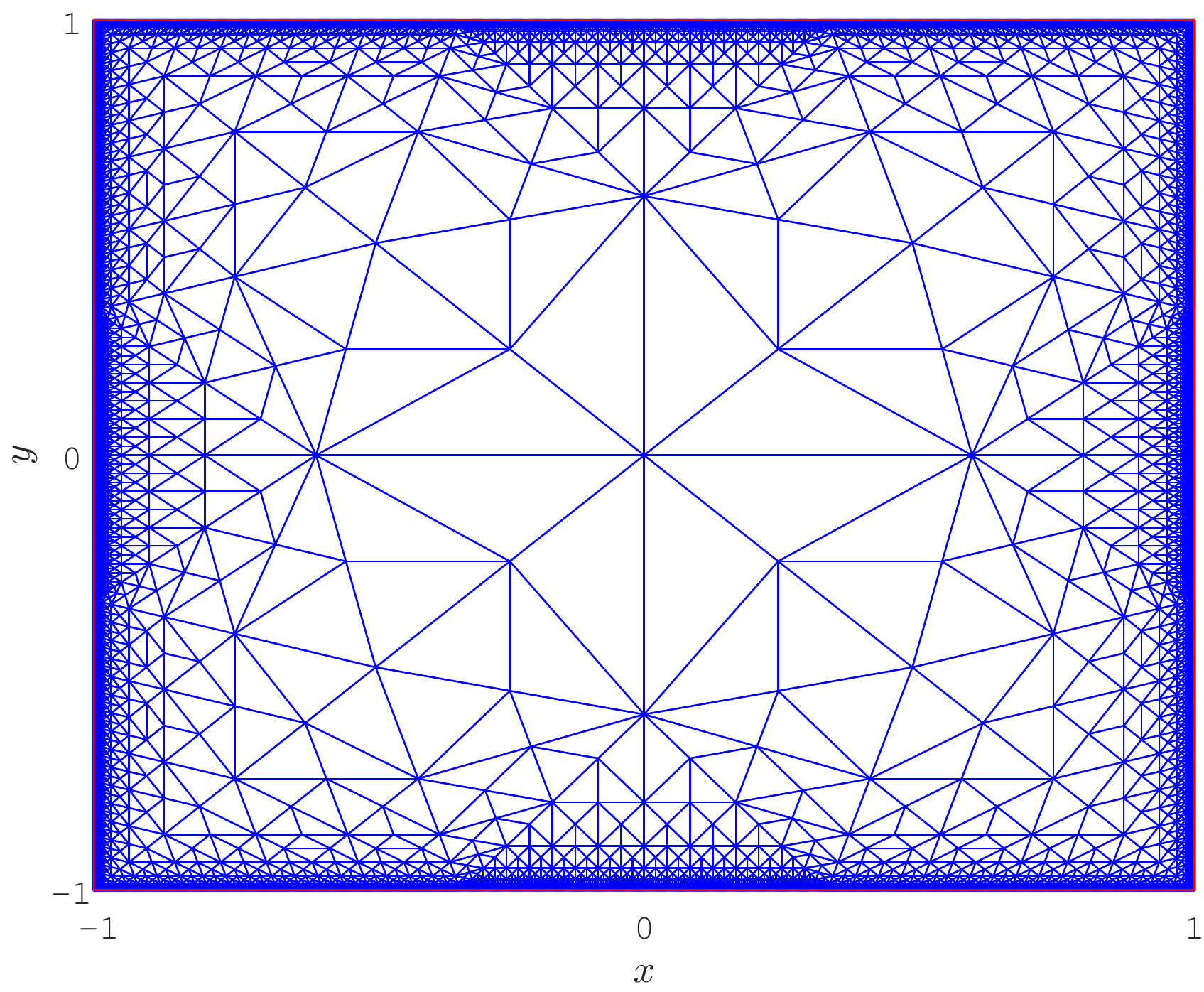}
\caption{Example~\ref{ex:2} for $\varepsilon =10^{-7}$: Initial mesh (left), and the adaptively refined mesh resolvong the solution (right).}		
\label{Ginzburg-Landau-Mesh}
\end{figure}
\end{example}

\section{Conclusions}\label{sc:concl}
The aim of this paper was to introduce a reliable and computationally feasible procedure for the numerical solution of semilinear elliptic boundary value problems, with possible singular perturbations. The key idea is to combine adaptive step size control for the \ptc-metod with an automatic mesh refinement finite element procedure. Furthermore, the sequence of linear problems resulting from the application of pseudo transient continuation and Galerkin discretization is treated by means of a robust (with respect to the singular perturbations) {\em a posteriori} residual analysis and a corresponding adaptive mesh refinement process. Our numerical experiments clearly illustrate the ability of our approach to reliably find solutions reasonably close to the initial guesses, and to robustly resolve the singular perturbations at an optimal rate.

\bibliographystyle{amsplain}
\bibliography{references}
\end{document}